\documentclass[11pt,a4paper, reqno]{amsart}
\pdfoutput=1
\usepackage{amssymb}
\usepackage{amsmath}
\usepackage{amsfonts}
\usepackage{mathrsfs}
\usepackage{bbm}
\usepackage{mathtools}
\usepackage{stmaryrd}
\usepackage[all]{xy}
\usepackage[top=20truemm,bottom=20truemm,left=20truemm,right=20truemm]{geometry}
\usepackage{amsthm}
\theoremstyle{definition}
\newtheorem{dfn}{Definition}
\newtheorem{thm}[dfn]{Theorem}
\newtheorem{pro}[dfn]{Proposition}

\newtheorem{ex}[dfn]{Example}

\usepackage{indentfirst}

\DeclareMathOperator{\kp}{KP} 
\DeclareMathOperator{\rep}{Rep} 
\DeclareMathOperator{\hilb}{Hilb} 
\DeclareMathOperator{\mor}{Mor} 
\DeclareMathOperator{\id}{id}
\DeclareMathOperator{\R}{R}

\title{Kac--Paljutkin Quantum Group as a Quantum Subgroup of the quantum $SU(2)$}
\author{Megumi Kitagawa}

\begin{document}
\email{g1770603@edu.cc.ocha.ac.jp}
\address{Department of Mathematics, Ochanomizu University,
Otsuka 2-1-1, Bynkyo-ku, 112-8610 Tokyo, Japan}

\begin{abstract}
We show that the Kac--Paljutkin Hopf algebra appears as a quotient of $C(SU_{-1}(2))$, which means that
the corresponding quantum group $G_{\kp}$ can be regarded as a quantum subgroup of $SU_{-1}(2)$.
We combine the fact that corepresentation category of the Kac--Paljutkin Hopf algebra is a 
Tambara--Yamagami tensor category associated with the Krein 4-group
and the method of graded twisting of Hopf algebras, to construct the Hopf $*$-homomorphism.
\end{abstract}

\date{February 20, 2019}

\maketitle

%----------------------------------------------------------------------------------------------------------------------------------------------
%Intro
%----------------------------------------------------------------------------------------------------------------------------------------------
\section{Introduction}
The theory of operator algebras emerged from quantum mechanics in the study of Murray and von Neumann 
to give it a mathematical framework. 
In this theory, the Gelfand--Naimark theorem states that a commutative operator algebra (unital C$^*$-algebra) 
is isomorphic to the algebra of continuous functions on some compact space.
According to this fundamental result, a general operator algebra can be seen as an algebra of continuous functions on 
a hypothetical ``noncommutative'' compact topological space. 

The next step is constructing additional structure to such objects. 
One example is the theory of compact quantum groups carried out by Woronowicz, 
which gives additional ``group structure'' on such spaces \cite{zbMATH04020060}. 
The most important example is the quantum $SU_q(2)$, constructed in \cite{zbMATH04107479}.
This gives a one parameter deformation of the algebra of functions on the compact group $SU(2)$ as 
a C$^*$-algebra with the coproduct, which explains the group law on the noncommutative space $SU_q(2)$.
The case of $q = 1$ corresponds to the algebra of $SU(2)$.
Parallel to Woronowicz's work, Drinfeld and Jimbo defined $q$-deformation of semisimple Lie groups, 
which are new Hopf algebras, by deforming universal enveloping algebras of semisimple Lie groups 
through the algebraic study of quantum integrable systems.
The C$^*$-algebra of $SU_q(2)$ contains a dense Hopf $*$-algebra of matrix coefficients of unitary corepresentations,
which can be regarded as the Hopf dual of the Drinfeld--Jimbo $q$-deformation Hopf algebra.

\smallskip
A significant feature of Woronowicz's construction is that the negative range $q < 0$ is allowed, 
which is different from naively setting $q$ to be a negative number in the Drinfeld--Jimbo construction.
In particular, a concrete description for $q=-1$ is given by Zakrzewski's realization of $C(SU_{-1}(2))$ as 
a C$^*$-subalgebra of $M_2(C(SU(2)))$ \cite{zbMATH04196226}.
This technique is also useful for computation in K-theory of algebra $C(SU_{-1}(2))$ \cite{MR3327485}.

In this paper, we are interested in the quantum subgroups of $SU_{-1}(2)$. 
An pioneering study was carried out by Podle\'{s} \cite{zbMATH00763767}, 
who investigated the subgroups and the quotient spaces of quantum $SU(2)$. 
The complete classification of quantum homogeneous spaces over $SU_q(2)$ realized as coideals is obtained by Tomatsu \cite{TOMATSU20081},
inspired by Wassermann's classification of ergodic actions of $SU(2)$ \cite{Wassermann1988}.
Their results give classification in terms of graphs.
According to the McKay correspondence, the homogeneous spaces for $SU(2)$ is classified by the extended Dynkin diagrams, 
and their results provide very similar picture. 

By Woronowicz's Tannaka--Krein duality \cite{zbMATH04086607}, a compact quantum group $G$ can be recovered from
their representation category $\rep G$, the C$^*$-tensor category of finite unitary representations, and the fiber functor.
The $SU_q(2)$ is special in this respect, because its representation category has a universality for 
the fundamental representation and the associated morphism which solves its conjugate equations.
The Tannaka--Krein duality for compact quantum homogeneous spaces over a compact quantum group $G$, 
established by De Commer and Yamashita, says that such homogeneous spaces correspond to module C$^*$-categories over $\rep G$.
Such module categories can be also described in terms of tensor functors from $\rep G$ to a category of bi-graded Hilbert spaces \cite{MR3121622},
the quantum $SU(2)$ case explained in detail in \cite{MR3420332}.
The universality of the representation category of $SU_q(2)$ implies that the quantum homogeneous spaces over $SU_q(2)$ 
are classified by graphs, generalizing the McKay correspondence.

\medskip
The Kac--Paljutkin Hopf algebra was introduced by Kac and Paljutkin as the smallest example of semisimple Hopf algebras 
which is neither commutative (function algebra of finite group) nor cocommutative (group algebra of finite group) \cite{MR0208401}.
In this paper we show that this algebra appears as a quotient of $C(SU_{-1}(2))$.
Conceptually, the corresponding quantum group $G_{\kp}$ can be regarded as a quantum subgroup of $SU_{-1}(2)$, 
and quotient map of Hopf algebras is ``restriction'' of functions.

A key fact for us is the corepresentation category of the Kac--Paljutkin algebra can be realized as 
a Tambara--Yamagami tensor category \cite{TAMBARA1998692} associated with the Krein 4-group, $K_4 = \mathbb{Z}/2\mathbb{Z} \times \mathbb{Z}/2\mathbb{Z}$.

We use the graded twist method of Bichon--Neshveyev--Yamashita \cite{zbMATH06679665} as another crucial technique to obtain Hopf $*$-homomorphism from $C(SU_{-1}(2))$.
This twisting gives a useful description of the Hopf algebra $C(SU_{-1}(2))$ as a deformation of the Hopf algebra $C(SU(2))$, suited to study of its Hopf quotients.
We apply their method for describing quantum subgroups of a compact quantum group obtained as the graded twisting of a genuine compact group.

\smallskip
This paper is organized as follows.
Section 2 is a preliminary section on compact quantum groups, the Tambara--Yamagami tensor categories, and the graded twisting of Hopf algebras. 
We also recall a presentation of the Kac--Paljutkin algebra following Tambara--Yamagami. 
Lastly in this section, we describe the construction of the graded twisting of Hopf algebras and then recall 
that the Hopf algebra $C(SU_{-1}(2))$ is isomorphic to the graded twisting of $C(SU(2))$.
In Section 3, we give a realization of Kac--Paljutkin Hopf algebra as a quotient of $C(SU_{-1}(2))$.
An essential ingredient in our computation is comparison of the two different kinds of projective representations of the Krein 4-group.
In Section 4, we describe the associated coideal which is one of the type $D_4^*$ discussed in \cite{TOMATSU20081},
which is also suggested in \cite{zbMATH00763767}.

\section{Preliminaries}
%----------------------------------------------------------------------------------------------------------------------------------------------
%Compact quantum groups
%----------------------------------------------------------------------------------------------------------------------------------------------
\subsection{Compact quantum groups}
For general theory of compact quantum groups, we refer to \cite{MR3204665}.
When we deal with C$^*$-algebras, the symbol $\otimes$ denotes the minimal tensor product.

\begin{dfn}
 A compact quantum group is a pair $(A, \Delta)$ of an unital C$^*$-algebra $A$ and an unital $*$-homomorphism 
 $\Delta \colon A \to A \otimes A$ called comultiplication such that
 	\begin{enumerate}
	 \item (coassociativity) $(\Delta \otimes \iota ) \Delta = (\iota \otimes \Delta) \Delta$,
	 \item (cancelation property) the spaces
	 \begin{align*}
	  (A \otimes 1) \Delta (A) = \text{span}\{ ( a \otimes 1) \Delta (b) | a,b \in  A \}, \
	  (1 \otimes A) \Delta(A) =  \text{span}\{ ( 1 \otimes a) \Delta (b) | a,b \in  A \}
	 \end{align*}
	 are dense in $A \otimes A$.
	\end{enumerate}
\end{dfn}

\begin{ex} 
\textup{
Let $G$ be a compact group. Then a compact quantum group $(A, \Delta)$ can be constructed as follows.
An unital C$^*$-algebra $A$ is the algebra $C(G)$ of complex valued continuous functions on $G$. 
In this case $A \otimes A$ can be identified with $C(G \times G)$ so the comultiplication 
$\Delta \colon C(G) \to C( G \times G)$ is given by
\[
 \Delta(f)(g,h) = f(gh) \ \text{for all} \ f \in C(G), \ g,h \in G.
\]
 }
\end{ex}
Any compact quantum group $(A, \Delta)$ with abelian $A$ is of this form.
Therefore we write  $(C(G), \Delta)$  for any compact quantum group $G$.

\begin{dfn} 
 \label{SUq2}
Let $q$ be a real number, $|q| \leq 1$, and $q \neq 0$.
The quantum $SU(2)$ group $SU_q(2)$ is defined as follows.
The algebra $C(SU_q(2))$ is the universal C$^*$-algebra generated by $\alpha$ and $\gamma$ such that 
\[
(u^q_{ij})_{i,j} = 
\begin{pmatrix}
\alpha & -q\gamma^* \\
\gamma & \alpha^*
\end{pmatrix} \text{is unitary.}
\]
The comultiplication $\Delta$ is defined by
\[
\Delta (u^q_{ij}) = \sum_k u^q_{ik} \otimes u^q_{kj}.
\]
Explicitly, we can write this comultiplication as
\begin{eqnarray*}
\begin{aligned}
	\Delta(\alpha) = \alpha \otimes \alpha - q \gamma^* \otimes \gamma, \
	\Delta(\gamma) = \gamma \otimes \alpha + \alpha^* \otimes \gamma.
\end{aligned}
\end{eqnarray*} 
\end{dfn}
From now unless we want to be specific about $q$ we write $u_{ij}$ for $u^q_{ij}$.
If $q=1$ then we can get the usual compact group $SU(2)$. 
For $q \neq 1$ the compact quantum group $SU_q(2)$ can be considered as deformations of $SU(2)$.

\begin{dfn} 
 \label{AoF}
Let $F \in GL_n(\mathbb{C})$, $n \geq 2$, such that  $F \overline{F} = \pm I_n$, where
$\overline{F}$ denotes the matrix with complex conjugates in every entries of $F$. 
The algebra $C(O_F^+)$ is the universal C$^*$-algebra generated by $u_{ij}, 1 \leq i,j \leq n,$ such that
\[
U=(u_{ij})_{i,j} \ \text{is unitary and} \ U = F U^c F^{-1} \hspace{0.3cm} (U^c = (u_{ij}^*)_{i,j}).
\]
The comultiplication $\Delta \colon C(O_F^+) \to C(O_F^+) \otimes C(O_F^+)$ is defined by 
\[
\Delta (u_{ij}) = \sum_k u_{ik} \otimes u_{kj}.
\]
This compact quantum group $O_F^+$ is called the free orthogonal quantum group associated with $F$.
\end{dfn}
We note that $SU_q(2)$ is an example of the free orthogonal quantum group by taking the matrix
\[
F = 
\begin{pmatrix}
0 & - \text{sgn}(q) |q|^{\frac{1}{2}} \\
|q|^{-\frac{1}{2}} & 0
\end{pmatrix}.
\]

%----------------------------------------------------------------------------------------------------------------------------------------------
%Tambara--Yamagami tensor category
%----------------------------------------------------------------------------------------------------------------------------------------------
\subsection{Tambara--Yamagami tensor category}

One of the Tambara--Yamagami tensor categories \cite{TAMBARA1998692} arising from the
Klein $4$-group $K_4 = \mathbb{Z}/2\mathbb{Z} \oplus \mathbb{Z}/2\mathbb{Z}$ is realised as 
the category of representations of the Kac--Paljutkin Hopf algebra \cite{TAMBARA1998692}.
Let us recall that the elements in the notation of \cite{TAMBARA1998692} of $K_4 = \{ e,a,b,c \}$
satisfies the relations $a^2 = b^2= c^2 =e, \ ab=c,\ bc=a, ca=b$.
What we focus on is the tensor category $\mathcal{C}(\chi, \tau)$ corresponding to 
the nondegenerate symmetric bicharacter $\chi = \chi_c$ of $K_4$
which is defined by
\begin{align*}
\chi_c(a,a) = \chi_c(b,b) = -1,\ \chi_c(a,b) =1,
\end{align*}
and the parameter $\tau = \frac{1}{2}$ satisfying $\tau^2 = \frac{1}{| K_4 |}$.
Its objects are finite direct sums of elements in $ S = K_4 \cup \{ \rho \}$.
Sets of morphisms between elements in $S$ are given by
\[
\mor (s,s') = 
\begin{cases}
	\mathbb{C} & s = s', \\
	0 & s \neq s',
\end{cases} 
\]
so $S$ is the set of irreducible classes of $\mathcal{C}(\chi, \tau)$.
Tensor products of elements in $S$ are given by  
\[
s \otimes t = st,\ s \otimes \rho = \rho = \rho \otimes s, \ \rho \otimes \rho = \bigoplus_{s \in K_4} s, \ (s,t \in K_4)
\]
and the unit object is $e$.
Associativities $\varphi$ are given by 
\begin{align*}
	\varphi_{s,t,u} &= \text{id}_{stu},&
	\varphi_{s,t,\rho} &= \varphi_{\rho,s,t} = \text{id}_{\rho},\\
	\varphi_{s,\rho,t} &=  \chi_c (s,t) \text{id}_{t}, &
	\varphi_{s,\rho,\rho} &= \varphi_{\rho,\rho,s} = \bigoplus_{k \in K_4} \text{id}_{k},\\
	\varphi_{\rho,s,\rho} &= \bigoplus_{k \in K_4} \chi_c (s,t) \text{id}_{k}, &
	\varphi_{\rho, \rho, \rho} &= \left( \frac{1}{2} \chi_c (k,l)^{-1} \text{id}_{\rho} \right)_{k,l} \colon \bigoplus_{k \in K_4} \rho \to \bigoplus_{l \in K_4} \rho,
\end{align*}
for $s,t,u \in K_4$.
This category is identified with the representation category of Kac--Paljutkin quantum group $G_{\kp}$, that is
\[
\mathcal{C} \left( \chi_c, \frac{1}{2} \right) \simeq \rep(G_{\kp}) 
\]
as tensor categories.
Here $(C(G_{\kp}), \Delta)$ is the Kac--Paljutkin algebra, that is, the eight dimensional Hopf algebra  
which is the noncommutative and noncocommutative algebra.
It is given by
\[
C(G_{\kp}) = \mathbb{C} \cdot {\epsilon} \oplus \mathbb{C} \cdot {\alpha} \oplus 
\mathbb{C} \cdot {\beta} \oplus \mathbb{C} \cdot {\gamma} \oplus M_2(\mathbb{C}),
\]
as an ($*$-)algebra.
The comultiplication $\Delta \colon C(G_{\kp}) \to C(G_{\kp}) \otimes C(G_{\kp})$ is defined by
 
\begin{align}
\begin{aligned}
  \Delta(\epsilon)
  ={}&
  \epsilon \otimes \epsilon + \alpha \otimes \alpha + \beta \otimes \beta + \gamma \otimes \gamma 
  + \frac{1}{2} \sum_{1 \leq i,j \leq 2} \epsilon_{ij} \otimes \epsilon_{ij},  \\
 \Delta(\alpha)
  ={}&
  \epsilon \otimes \alpha + \alpha \otimes \epsilon + \beta \otimes \gamma + \gamma \otimes \beta 
  \\&+
 \frac{1}{2} (\epsilon_{11} \otimes \epsilon_{22} + i \epsilon_{12} \otimes \epsilon_{21}
			- i \epsilon_{21} \otimes \epsilon_{12} + \epsilon_{22} \otimes \epsilon_{11}),  \\
  \Delta(\beta)
  ={}&
   \epsilon \otimes \beta + \beta \otimes \epsilon + \alpha \otimes \gamma + \gamma \otimes \alpha 
  \\&+
 \frac{1}{2} (\epsilon_{11} \otimes \epsilon_{22} - i \epsilon_{12} \otimes \epsilon_{21}
			+ i \epsilon_{21} \otimes \epsilon_{12} + \epsilon_{22} \otimes \epsilon_{11}), \\
  \Delta(\gamma) 
  ={}&
  \epsilon \otimes \gamma + \gamma \otimes \epsilon +\alpha \otimes \beta + \beta \otimes \alpha 
  \\&+
 \frac{1}{2} (\epsilon_{11} \otimes \epsilon_{11} - \epsilon_{12} \otimes \epsilon_{12}
			- \epsilon_{21} \otimes \epsilon_{21} + \epsilon_{22} \otimes \epsilon_{22}),  \\
  \Delta(x)
  ={}&
   \epsilon \otimes x +\alpha \otimes u_{\alpha} x u_{\alpha}^* 
		+ \beta \otimes u_{\beta} x u_{\beta}^* + \gamma \otimes  u_{\gamma} x u_{\gamma}^* 
  \\&+
 x \otimes \epsilon + \overline{u}_{\alpha} x \overline{u}_{\alpha}^* \otimes \alpha 
		+ \overline{u}_{\beta} x \overline{u}_{\beta}^* \otimes \beta + \overline{u}_{\gamma} x \overline{u}_{\gamma}^* \otimes \gamma,
\label{coprod_KP}
\end{aligned}
\end{align}
for projections $ \epsilon, \alpha, \beta, \gamma$ and $x \in M_2(\mathbb{C})$,
where $\epsilon_{ij}$ are the matrix units in $ M_2(\mathbb{C})$ and 
\begin{align*}
u_{\alpha} =
\begin{pmatrix}
0 & i \\
1 & 0
\end{pmatrix}, \
u_{\beta} =
\begin{pmatrix}
0 & 1 \\
i & 0
\end{pmatrix}, \
u_{\gamma} =
\begin{pmatrix}
-1 & 0 \\
0 & 1
\end{pmatrix}.
\end{align*}

 %----------------------------------------------------------------------------------------------------------------------------------------------
%graded twist
%----------------------------------------------------------------------------------------------------------------------------------------------
\subsection{Graded twisting of Hopf algebras} 
Let us describe the graded twist construction \cite{zbMATH06679665}.
The algebra $C(SU(2))$ of continuous functions on the compact group $SU(2)$ can be identified with the space 
$\bigoplus_{n \in \mathbb{N}} \overline{H}_{n/2} \otimes H_{n/2}$,
where $\overline{H}_{n/2} \otimes H_{n/2}$ denotes matrix coefficients of the irreducible representation of $SU(2)$ of dimension $n+1$.
Half integers $\{ n/2 \}_{n \in \mathbb{N}}$ can be divided into integers $\{0,1,2,3,\ldots \}$ for $n$ even and 
others $\{1/2,3/2,5/2, \ldots \}$ for $n$ odd. Thus the space above can be decomposed as 
\begin{align*}
(\bigoplus_{n \colon \text{even}} \overline{H}_{n/2} \otimes H_{n/2}) \oplus (\bigoplus_{n \colon \text{odd}} \overline{H}_{n/2} \otimes H_{n/2}).
\end{align*}
The component with even $n$ forms the algebra of continuous functions on $SO(3)$,
so we denote the whole space by $C(SO(3)) \oplus C(SU(2))_{\text{odd}}$.
Let $\{ e_1, e_2 \}$ be an orthonormal basis of $H_{1/2}$.
Unit vectors $\overline{e}_i \otimes e_j$ in $\overline{H}_{1/2} \otimes H_{1/2} \subset C(SU(2))_{\text{odd}}$ are denoted by $u_{ij}$.

Consider an action $\alpha$ of the group $\mathbb{Z}/2\mathbb{Z}$ on the Hopf algebra $C(SU(2))$ defined by
\begin{align}
 \alpha_g \left \lgroup \begin{pmatrix} 
		u_{11} & u_{12} \\
		u_{21} & u_{22} 
		\end{pmatrix} \right \rgroup
  ={}&
  \begin{pmatrix} 
	u_{11} & - u_{12} \\
	- u_{21} & u_{22} 
 \end{pmatrix} \label{action}
  \\={}&
 \begin{pmatrix} 
	i & 0 \\
	0 & -i 
 \end{pmatrix} 
 			\begin{pmatrix} 
			u_{11} & u_{12} \\
			u_{21} & u_{22} 
			\end{pmatrix}		
						\begin{pmatrix} 
						-i & 0 \\
						0 & i 
						\end{pmatrix} \notag
\end{align}
for the generator $g$ of $\mathbb{Z}/2\mathbb{Z}$.

Next take the crossed product $C(SU(2)) \rtimes_{\alpha} \mathbb{Z}/2\mathbb{Z}$ of the Hopf algebra. 
Define the graded twisting of $C(SU(2))$ by $\alpha$ as the subalgebra of crossed product 
\[
C(SU(2))^{t,\alpha} = C(SO(3)) \oplus (C(SU(2))_{\text{odd}} \cdot \lambda_g) \subset C(SU(2)) \rtimes_{\alpha} \mathbb{Z}/2\mathbb{Z}.
\] 
Generators $u_{ij} \lambda_g$ in $C(SU(2))^{t,\alpha}$ are denoted by $u_{ij}'$. 
It follows that the matrix $(u'_{ij})_{i,j=1}^2$ becomes unitary because $(u_{ij})_{i,j=1}^2$ is an unitary matrix.
They satisfy the same relations as the generators of $C(SU_{-1}(2))$. Indeed,
\begin{align*}
(u_{11}')^* ={}& (u_{11} \lambda_g)^* = \lambda_{g^{-1}} u_{11}^* = \lambda_g u_{22} = u_{22} \lambda_g = u_{22}',\\
(u_{21}')^* ={}& (u_{21}\lambda_g)^* = \lambda_{g^{-1}} u_{21}^* = \lambda_g (- u_{21}) = u_{12} \lambda_g = u_{12}',
\end{align*}
so that the matrix $(u'_{ij})_{i,j=1}^2$ can be described as
\[
\begin{pmatrix}
 u_{11}' & u_{12}' \\ 
 u_{21}' & u_{22}'
 \end{pmatrix}
 = \begin{pmatrix}
 u_{11}' & (u_{21}')^* \\ 
 u_{21}' & (u_{11}')^*
 \end{pmatrix}.
 \]
Moreover, images of $u_{ij}'$ by the comlitiplication $\Delta$ on $C(SU(2))^{t,\alpha}$ are given by
\begin{align}
\Delta(u_{ij}') = (\sum_{k} u_{ik} \otimes u_{kj}) \lambda_g \otimes \lambda_g \label{coprod} 
 = \sum_{k} u_{ik}' \otimes u_{kj}'
\end{align}
Thus we obtain an isomorphism of Hopf algebras
\[
C(SU(2))^{t,\alpha} \simeq C(SU_{-1}(2))
\]
by mapping $u_{ij}'$ in $C(SU(2))^{t,\alpha}$ to $u_{ij}^{(-1)}$ in $C(SU_{-1}(2))$.\\

%----------------------------------------------------------------------------------------------------------------------------------------------
%Realization of Kac--Paljutkin Hopf algebra as a quotient
%----------------------------------------------------------------------------------------------------------------------------------------------
\section{Realization of Kac--Paljutkin Hopf algebra as a quotient}
Besides the formulation of graded twist, \cite{zbMATH06679665} provided a method for describing quantum subgroups 
of a compact quantum groups obtained as the graded twisting of a genuine compact group.
Let us apply their method to our compact quantum group $SU_{-1}(2)$. 

\begin{pro}[{[2, Example 4.11]}]
A quantum subgroup of $SU_{-1}(2)$ with noncommutative function algebra
corresponds to a closed subgroup of $SU(2)$ containing $\{ \pm I_2 \}$, 
being stable under the $\mathbb{Z}/2\mathbb{Z}$-action defined in (\ref{action}) and containing an element 
\begin{align}\label{abcd}
\begin{pmatrix}
 a & b \\ 
 c & d
 \end{pmatrix}
\text{with} \ abcd \neq 0.
\end{align}
\end{pro}

Take a subgroup $\tilde{V}$ of $SU(2)$ generated by
\[
s_1 = \frac{1}{\sqrt{2}} \begin{pmatrix}
 					i & i \\ 
 					i & -i
				 \end{pmatrix},\
s_2 = \frac{1}{\sqrt{2}} \begin{pmatrix}
 					-i & i \\ 
 					i & i
				 \end{pmatrix},\
s_3 = \begin{pmatrix}
 0 & -1 \\ 
 1 & 0
 \end{pmatrix}.
\]
They satisfy relations $s_i^2 = - I_2$ for $i = 1,2,3$ and $s_1 s_2 = s_3 = -s_2 s_1$.
It is the group of order eight with elements $\pm s_i$ for $i = 1,2,3$ and $\pm I_2$.
This subgroup $\tilde{V}$ is related to the Klein 4-group $K_4 $ via an isomorphism 
$\tilde{V} / \{ \pm I_2 \} \simeq K_4$.
This subgroup satisfies the conditions mentioned in the above proposition. Namely,
$s_i^2 = - I_2$ for $i = 1,2,3$ so $\tilde{V}$ has the elements $ \pm I_2$.
Since $\alpha_g$ transforms the elements 
\begin{align*}
s_1 \mapsto - s_2, \ s_2 \mapsto -s_1, s_3 \mapsto - s_3, I_2 \mapsto I_2,
\end{align*}
$\tilde{V}$ is stable under the $\mathbb{Z}/2\mathbb{Z}$-action. 
The element $s_1$ in $\tilde{V}$ gives an example of an element of the form in (\ref{abcd}).

Consider the graded twisting $C(\tilde{V})^{t,\alpha} = C(\tilde{V})_{\text{even}} \oplus (C(\tilde{V})_{\text{odd}} \cdot \lambda_g)$ of 
the Hopf algebra $C(\tilde{V})$ of continuous functions on $\tilde{V}$. 

We are now ready to state our main result.

\begin{thm}
There exists a surjective Hopf $*$-homomorphism from $C(SU_{-1}(2))$ to $C(G_{\kp})$.
It can be constructed by a composition of 
the Hopf $*$-isomorphism $C(SU_{-1}(2)) \to C(SU(2))^{t,\alpha}$,
a surjective Hopf $*$-homomorphism $C(SU(2))^{t,\alpha} \to C(\tilde{V})^{t,\alpha}$,
and the Hopf $*$-isomorphism $C(\tilde{V})^{t,\alpha} \to C(G_{\kp})$ defined by 
\begin{align*}
(\epsilon, \alpha', \beta', \gamma') & \mapsto (\epsilon, \gamma, \alpha, \beta), \\
M_2(\mathbb{C}) \ni x & \mapsto vxv^* \quad
\left( v = \begin{pmatrix}
-1 & 0 \\
0 & i
\end{pmatrix} \right).
\end{align*}
\end{thm}

There is a surjective homomorphism $C(SU(2))^{t, \alpha} \to C(\tilde{V})^{t,\alpha}$, 
it represents a quantum subgroup of $SU_{-1}(2)$.

\begin{pro}
The Hopf $*$-algebra $C(\tilde{V})^{t,\alpha}$ is  noncommutative and noncocommutative.
\end{pro}

\begin{proof}
Indeed, if we take a product of an element $\delta_{s_1} + \delta_{-s_1}$ in $C(\tilde{V})_{\text{even}}$ and an element 
$(\delta_{s_2} - \delta_{-s_2}) \lambda_g$ in $C(\tilde{V})_{\text{odd}} \cdot \lambda_g$ in this order, then we have
\[
(\delta_{s_1} + \delta_{-s_1})(\delta_{s_2} - \delta_{-s_2}) \lambda_g = 0.
\] 
On the other hand, 
\[
(\delta_{s_2} - \delta_{-s_2}) \lambda_g (\delta_{s_1} + \delta_{-s_1}) 
= (\delta_{s_2} - \delta_{-s_2}) (\delta_{-s_2} + \delta_{s_2}) \lambda_g
=  (\delta_{s_2} - \delta_{-s_2}) \lambda_g,
\]
which shows noncommutativity of $C(\tilde{V})^{t,\alpha}$.
Furthermore, the coproduct on $C(\tilde{V})^{t,\alpha}$ is induced by 
\[
\Delta(\delta_h) = \sum_{h = k_1 k_2} \delta_{k_1} \otimes \delta_{k_2}
\]
for $h \in \tilde{V}$ and (\ref{coprod}).
Using it we can see that the comultiplication $\Delta$ on $C(\tilde{V})^{t,\alpha}$ is noncocommutative by observing that 
$\Delta( (\delta_{s_3} - \delta_{-s_3}) \lambda_g ) \neq \Delta^{\text{op}}( (\delta_{s_3} - \delta_{-s_3}) \lambda_g )$
for an element $(\delta_{s_3} - \delta_{-s_3}) \lambda_g$ in $C(\tilde{V})^{t,\alpha}$. 
By direct computation we get
\begin{align*}
\Delta( (\delta_{s_3} - \delta_{-s_3}) \lambda_g )  ={}& (\Delta(\delta_{s_3} ) - \Delta(\delta_{-s_3}) ) (\lambda_g \otimes \lambda_g)
\\={}& \{ (\delta_{s_1} -\delta_{-s_1}) \otimes (\delta_{s_2} - \delta_{-s_2}) -(\delta_{s_2} - \delta_{-s_2}) \otimes (\delta_{s_1} -\delta_{-s_1}) 
\\&{}+ (\delta_{s_3} -\delta_{-s_3}) \otimes (\delta_{I_2} -\delta_{-I_2}) + (\delta_{I_2} -\delta_{-I_2}) \otimes (\delta_{s_1} -\delta_{-s_1}) \}
(\lambda_g \otimes \lambda_g) ,
\end{align*}
and 
\begin{align*}
\Delta^{\text{op}}( (\delta_{s_3} - \delta_{-s_3}) \lambda_g )  ={}&
 \{ (\delta_{s_2} - \delta_{-s_2}) \otimes (\delta_{s_1} -\delta_{-s_1}) - (\delta_{s_1} -\delta_{-s_1}) \otimes (\delta_{s_2} - \delta_{-s_2})
\\&{}+ (\delta_{s_3} -\delta_{-s_3}) \otimes (\delta_{I_2} -\delta_{-I_2}) + (\delta_{I_2} -\delta_{-I_2}) \otimes (\delta_{s_1} -\delta_{-s_1})\}
(\lambda_g \otimes \lambda_g) .
\end{align*}
This concludes the proof.
\end{proof}

\begin{proof}[Proof of Theorem 6]
The only thing we need to describe is a concrete isomorphism of Hopf algebras 
$(C(\tilde{V})^{\text{t},\alpha}, \Delta_{\text{gr}}) \to (C(G_{\kp}), \Delta_{G_{\kp}})$. 
We set generators in $C(\tilde{V})^{t,\alpha} = C(\tilde{V})_{\text{even}} \otimes ( C(\tilde{V})_{\text{odd}} \cdot \lambda_g)$ 
regarded as elements in 
\[
\mathbb{C} \cdot {\epsilon} \oplus \mathbb{C} \cdot {\alpha'} \oplus 
\mathbb{C} \cdot {\beta'} \oplus \mathbb{C} \cdot {\gamma'} \oplus M_2(\mathbb{C})
\]
for projections $\epsilon, \alpha', \beta'$, and $\gamma'$ such that the triplet $(\alpha', \beta', \gamma')$ is obtained from 
permutation of the triplet $(\alpha, \beta, \gamma)$ by the following mapping
\begin{align*}
\delta_{s_1} + \delta_{-s_1} \mapsto
\begin{pmatrix}
 1 & 0 \\ 
 0 & 0
\end{pmatrix}, & \quad
\delta_{s_2} + \delta_{-s_2} \mapsto
\begin{pmatrix}
 0 & 0 \\ 
 0 & 1
\end{pmatrix}, \\
\delta_{s_3} + \delta_{-s_3} \mapsto \beta' + \gamma', & \quad
\delta_{I_2} + \delta_{-I_2} \mapsto \epsilon' + \alpha',\\
(\delta_{s_1} - \delta_{-s_1}) \lambda_g \mapsto
\begin{pmatrix}
 0 & -1 \\ 
 0 & 0
\end{pmatrix}, & \quad
(\delta_{s_2} - \delta_{-s_2}) \lambda_g \mapsto
\begin{pmatrix}
 0 & 0 \\ 
 1 & 0
\end{pmatrix}, \\
(\delta_{s_3} - \delta_{-s_3}) \lambda_g \mapsto i (\beta' - \gamma'), & \quad
(\delta_{I_2} - \delta_{-I_2} ) \lambda_g \mapsto \epsilon' - \alpha'.
\end{align*}
Applying this mapping, we can see that formulas for images of elements in $C(\tilde{V})^{t,\alpha}$ 
by the comultiplication $\Delta_{\text{gr}}$ is given by
\begin{align}
\begin{aligned}
\Delta_{\text{gr}}(\epsilon) ={}& \epsilon \otimes \epsilon + \alpha' \otimes \alpha' + \beta' \otimes \beta' + \gamma' \otimes \gamma'  \\
&+ \frac{1}{2} (\epsilon_{11} \otimes \epsilon_{11} - \epsilon_{12} \otimes \epsilon_{12} 
	- \epsilon_{21} \otimes \epsilon_{21} + \epsilon_{22} \otimes \epsilon_{22} ),   \\
\Delta_{\text{gr}}(\alpha') ={}& \epsilon \otimes \alpha'  + \alpha' \otimes \epsilon + \beta' \otimes \gamma' + \gamma' \otimes \beta'  \\
&+ \frac{1}{2} (\epsilon_{11} \otimes \epsilon_{11} + \epsilon_{12} \otimes \epsilon_{12} 
	+ \epsilon_{21} \otimes \epsilon_{21} + \epsilon_{22} \otimes \epsilon_{22} ), \\
\Delta_{\text{gr}}(\beta') ={}& \epsilon \otimes \beta'  + \beta' \otimes \epsilon + \alpha' \otimes \gamma' + \gamma' \otimes \alpha'  \\
&+ \frac{1}{2} (\epsilon_{11} \otimes \epsilon_{22} + i \epsilon_{12} \otimes \epsilon_{21} 
	-i \epsilon_{21} \otimes \epsilon_{12} + \epsilon_{22} \otimes \epsilon_{11} ), \\
\Delta_{\text{gr}}(\gamma') ={}& \epsilon \otimes \gamma'  + \gamma' \otimes \epsilon + \alpha' \otimes \beta' + \beta' \otimes \alpha'  \\
&+ \frac{1}{2} (\epsilon_{11} \otimes \epsilon_{22} - i \epsilon_{12} \otimes \epsilon_{21} 
	+ i \epsilon_{21} \otimes \epsilon_{12} + \epsilon_{22} \otimes \epsilon_{11} ), \\	
  \Delta_{\text{gr}}(x)
  ={}&
 \epsilon \otimes x + \alpha' \otimes w_{\alpha'} x w_{\alpha'}^* 
 + \beta' \otimes w_{\beta'} x w_{\beta'}^* + \gamma' \otimes w_{\gamma'} x w_{\gamma'}^* 
  \\&+
 x \otimes \epsilon + \overline{w}_{\alpha'} x \overline{w}_{\alpha'}^* \otimes \alpha'
 + \overline{w}_{\beta'} x \overline{w}_{\beta'}^* \otimes \beta' + \overline{w}_{\gamma'} x \overline{w}_{\gamma'}^* \otimes \gamma',
 \label{coprod_gr}
\end{aligned}
\end{align}
for projections $\epsilon, \alpha', \beta', \gamma'$ and $x \in M_2(\mathbb{C})$, where
$\epsilon_{ij}$ are the matrix units and 
\begin{align*}
w_{\alpha'} =
\begin{pmatrix}
 -1 & 0 \\ 
 0 & 1
\end{pmatrix}, \quad
w_{\beta'} =
\begin{pmatrix}
 0 & 1 \\ 
 i & 0
\end{pmatrix}, \quad
w_{\gamma'} =
\begin{pmatrix}
 0 & -i \\ 
 -1 & 0
\end{pmatrix}.
\end{align*} 
We can observe that these unitary matrices $w_{\alpha'}$, $w_{\beta'}$ and $w_{\gamma'}$ are transformed to unitary matrices
$u_{\gamma}$, $u_{\alpha}$ and $u_{\beta}$ in the formula of $\Delta_{G_{\kp}}$ respectively, 
by taking adjoint by a unitary matrix 
\begin{align} \label{unitaries}
v = \begin{pmatrix}
-1 & 0 \\
0 & i
 \end{pmatrix}, \
v w_{\alpha'} v^* = u_\gamma,
v w_{\beta'} v^* = u_\alpha, 
v w_{\gamma'} v^* = u_\beta.
\end{align}
Moreover, $\Delta_{\text{gr}}(\epsilon)$, $\Delta_{\text{gr}}(\alpha')$, $\Delta_{\text{gr}}(\beta')$ and $\Delta_{\text{gr}}(\gamma')$
coincide with $\Delta_{G_{\kp}}(\epsilon)$, $\Delta_{G_{\kp}}(\gamma)$, $\Delta_{G_{\kp}}(\alpha)$ 
and $\Delta_{G_{\kp}}(\beta)$ respectively, by 
$\Phi$
defined by
\begin{align*}
(\epsilon, \alpha', \beta', \gamma') \mapsto (\epsilon, \gamma, \alpha, \beta), \
M_2(\mathbb{C}) \ni x \mapsto vxv^*.
\end{align*}
Then (\ref{unitaries}) implies that $\Phi$ intertwines (\ref{coprod_gr}) to (\ref{coprod_KP}). 
For instance, 
$\alpha' \otimes w_{\alpha'} x w_{\alpha'}^*$ in (\ref{coprod_gr})
is transformed to $\gamma \otimes v w_{\alpha'} v^*  (v x v^*) v w_{\alpha'}^* v^*
= \gamma \otimes u_{\gamma} \text{Ad}_v (x) u_{\gamma}^*$ in (\ref{coprod_KP}).
Hence we obtain the isomorphism $(C(\tilde{V})^{\text{t},\alpha}, \Delta_{\text{gr}}) \to (C(G_{\kp}), \Delta_{G_{\kp}})$.
\end{proof}

%----------------------------------------------------------------------------------------------------------------------------------------------
%Representations of compact quantum groups
%----------------------------------------------------------------------------------------------------------------------------------------------
\section{$\rep SU_{-1}(2)$-module homomorphisms}
\subsection{Further preliminaries}

In the following we assume that $\mathcal{C}$ is a C$^*$-tensor category and $\mathbbm{1}$ denotes the unit object in $\mathcal{C}$.
For details, we refer to \cite{MR3204665}.

\begin{dfn}
A representation of a compact quantum group $G$ on a finite dimensional vector space $H_U$ is 
an invertible element $U \in B(H_U) \otimes C(G)$ such that 
\[
(\iota \otimes \Delta)(U) = U_{12} U_{13} \ \text{in} \ B(H_U) \otimes C(G) \otimes C(G).
\]
\end{dfn}
The representation $(U, H_U)$ is called unitary representation if $H_U$ is a Hilbert space and $U$ is unitary.
The unitaries $U = (u_{ij})_{i,j}$ in Definition \ref{SUq2} and Definition \ref{AoF} define unitary representations of each compact quantum group. 
They are called the fundamental representations of the corresponding quantum groups.

The 1-dimensional corepresentations of $(C(G_{\kp}), \Delta_{G_{\kp}})$ are following.
\begin{align}
\begin{aligned}
u_1 = \epsilon + \alpha + \beta + \gamma + 
 \begin{pmatrix}
 1 & 0 \\ 
 0 &  1
 \end{pmatrix}, \quad & 
 u_2 = \epsilon - \alpha - \beta + \gamma + 
 \begin{pmatrix}
 1 & 0 \\ 
 0 &  -1
 \end{pmatrix}\\ 
 u_3 = \epsilon + \alpha + \beta + \gamma + 
 \begin{pmatrix}
 -1 & 0 \\ 
 0 & -1
 \end{pmatrix}, \quad &
u_4 = \epsilon - \alpha - \beta + \gamma + 
 \begin{pmatrix}
 -1 & 0 \\ 
 0 & 1
 \end{pmatrix}. \label{1dim}
\end{aligned}
\end{align} 

The oriented graph with weights corresponding to the representation category $\rep (G_{\kp})$ is in Figure \ref{fig}.
Each vertex corresponds to an irreducible object in $\rep (G_{\kp})$ with labeling corresponding to the convention of Section 2.2. 
Total weights on the oriented edges starting from one vertex is equal to 2.
See \cite{MR3420332} for the interpretation of the weights of this graph.
\begin{figure}[hb]
\centerline{
\xymatrix@=3em{
&&\bullet^{a} \ar @/^/[d]^{^2} &&
\\ & \ar @/^/[r]^{_{2}} \bullet_{e} & \ar @/^/[l]^{^{1/2}} \ar @/^/[r]^{_{1/2}}  \bullet_{\rho}
 \ar @/^/[d]^{^{1/2}} \ar @/^/[u]^{^{1/2}} & \ar @/^/[l]^{^{2}}  \bullet_{b}&
\\ && \ar @/^/[u]^{^{2}} \bullet_{c} &&
}
}
\caption{$D_4^{(1)}$} 
\label{fig}
\end{figure}

\begin{dfn}
Assume $(U, H_U)$ and $(V, H_V)$ are finite dimensional representations of a compact quantum group $G$.
Then an operator $T \colon H_U \to H_V$ is an intertwiner from $U$ to $V$ if
\[
(T \otimes 1) U = V (T \otimes 1).
\]
\end{dfn}
The space of intertwiners from $U$ to $V$ is denoted by $\mor (U,V)$.
A representation $(U, H_U)$ is irreducible if $\mor (U,U) = \mathbb{C}$.

Let $U \in M_2(C(G_{\kp}))$ be the fundamental representation $u_i$ of $G_{\kp}$ (\ref{1dim}).
Its tensor product $U \otimes U$ decomposes into $\sum_{i=1}^4 P_i \otimes u_i$ wth
mutually orthogonal matrices
\begin{align*}
P_1 = \frac{1}{2}
 \begin{pmatrix}
 0 & 0 & 0 & 0 \\ 
 0 & 1 & 1 & 0 \\
 0 & 1 & 1 & 0 \\
 0 & 0 & 0 & 0 \\ 
 \end{pmatrix}, \quad &
P_2 = \frac{1}{4}
 \begin{pmatrix}
 1 & 1 & -1 & 1 \\
 1 & 1 & -1 & 1 \\
 -1 & -1 & 1 & -1 \\
 1 & 1 & -1 & 1 \\
 \end{pmatrix},\\
P_3 = \frac{1}{2}
 \begin{pmatrix}
 1 & 0 & 0 & -1 \\ 
 0 & 0 & 0 & 0 \\
 0 & 0 & 0 & 0 \\
 -1 & 0 & 0 & 1 \\ 
 \end{pmatrix},  \quad & 
 P_4 = \frac{1}{4}
 \begin{pmatrix}
 1 & -1 & 1 & 1 \\
 -1 & 1 & -1 & -1 \\
 1 & -1 & 1 & 1 \\
1 & -1 & 1 & 1 \\
 \end{pmatrix}.
\end{align*} 

%----------------------------------------------------------------------------------------------------------------------------------------------
%module category / homomorphisms
%----------------------------------------------------------------------------------------------------------------------------------------------
\begin{dfn}[e.g., \cite{MR3121622}]
Let $\mathcal{D}$ be a C$^*$-category. Then $(\mathcal{D}, M, \phi, e)$ is a left $\mathcal{C}$-module C$^*$-category
if $M$ is a bilinear $*$-functor $\mathcal{C} \times \mathcal{D} \to \mathcal{D}$ with natural transformations 
$\phi \colon M((- \otimes -), -) \to M(-, M(-, -))$ and $e \colon M(\mathbbm{1}, -) \to \text{id}$
satisfying certain coherence condition. 
We often abbreviate this left $\mathcal{C}$-module C$^*$-category as $\mathcal{D}$.

When we write $U \otimes X$ for $M(U, X)$, then the condition is described as the commutative diagrams below.
\begin{figure}[h]
\begin{minipage}[b]{0.45\textwidth}
\centerline{
\xymatrix@=1em{ 
(U \otimes V \otimes W ) \otimes X \ar[rrr]^>>>>>>>>{\phi_{U, V \otimes W, X}}  \ar[dd]_{\phi_{U \otimes V, W, X}}& 
&&U \otimes ((V \otimes W ) \otimes X) \ar[dd]^{\text{id}_U \otimes \phi_{V, W, X}}
\\ &&
\\ (U \otimes V) \otimes (W  \otimes X) \ar[rrr]^>>>>>>>>{\phi_{U \otimes V, W, X}} &&& U \otimes (V \otimes (W  \otimes X))
}}
\end{minipage}
\hspace{0.1cm}
\begin{minipage}[b]{0.45\textwidth}
\centerline{
\xymatrix@=1em{ 
& U \otimes (\mathbbm{1} \otimes X) \ar[dr]^>>{\text{id}_U \otimes e_X} &
\\ U \otimes X \ar[ur]^<<<{\phi_{U, \mathbbm{1}, X}} \ar[rr]^{\text{id}_{U \otimes X}} \ar[dr]_<<<<{\phi_{\mathbbm{1}, U, X}} && U \otimes X
\\ & \mathbbm{1} \otimes (U \otimes X) \ar[ur]_<<<<<<{e_{U \otimes X}} &
}}
\end{minipage}
\end{figure}
\end{dfn}

\begin{ex}
Let $G$ be a compact (quantum) group, and $H$ be a closed (quantum) subgroup of $G$.
Then $\rep H$ is a $\rep G$-module C$^*$-category. 
For $\pi \in \rep G$ and $\theta \in H$, $\pi \otimes \theta$ is defined by $\pi|_{H} \otimes \theta$.
The restriction functor $\rep G \to \rep H$ induces this module category.
We are particularly interested in the case of $G= SU_{-1}(2)$ and $H = G_{\kp}$.
\end{ex}

\begin{dfn}
Let $\mathcal{D}, \mathcal{D}'$ be module categories over a fixed C$^*$-tensor category $\mathcal{C}$.
Then $(G, \psi)$ is a $\mathcal{C}$-module homomorphism from $\mathcal{D} \to \mathcal{D}'$ if 
$G$ is a functor from $\mathcal{D} \to \mathcal{D}'$ and 
$\psi$ is a natural unitary equivalence $G(- \otimes -) \to - \otimes G -$ satisfying the commutative diagrams below.
\begin{figure}[h]
\begin{minipage}[b]{0.45\textwidth}
\centerline{
\xymatrix@=1em{ 
G(\mathbbm{1} \otimes X) \ar[rr]^{\psi_{\mathbbm{1}, X}}  \ar[dd]_{G(e)} &&
\mathbbm{1} \otimes GX \ar[lldd]^{e}
\\ &&
\\ GX &&
}}
\end{minipage}
\hspace{0.1cm}
\begin{minipage}[b]{0.45\textwidth}
\centerline{
\xymatrix@=1em{ 
& U \otimes G(V \otimes X) \ar[dr]^>>{\text{id}_U \otimes \psi_{V, X}} &
\\ G(U \otimes (V \otimes X)) \ar[d]_{G(\phi_{U, V, X})} \ar[ur]^<<<{\psi_{U, V \otimes X}} && U \otimes (V \otimes GX) \ar[d]^{\phi_{U, V, GX}}
\\ G((U \otimes V) \otimes X) \ar[rr]_{\psi_{U \otimes V, X}} && (U \otimes V) \otimes GX
}}
\end{minipage}
\end{figure}
\end{dfn}

A $\rep G$-module homomorphism for a compact quantum group $G$ corresponds to the Hopf homomorphism
which define the action of $G$.

%----------------------------------------------------------------------------------------------------------------------------------------------
%$RepSU_{-1}(2)$-module homomorphisms
%----------------------------------------------------------------------------------------------------------------------------------------------
\subsection{Concrete computation}
Unitary maps $\psi$ consisting a $\rep SU_{-1}(2)$-module homomorphism $(G, \psi)$ together with 
a functor $G \colon \rep G_{\kp} \to \hilb_f$ can be given
by solving equations provided by the interpretation of conditions on the natural equivalence $\psi$ 
in terms of bigraded vector spaces \cite{MR3121622,MR3420332}.\\
 
From the information in the graph in Figure \ref{fig} we can write up the $q$-fundamental solution in $\rep G_{\kp}$.
Recall that $\text{Irr}G_{\kp} = K_4 \cup \{ \rho \}$.
The bigraded vector spaces associated with the $\rep SU_{-1}(2)$-module category $\rep G_{\kp}$ are denoted by 
$H_{\rho g}$ and $H_{g \rho}$ for $g \in K_4= \{e,a,b,c \}$.
They are all one dimensional so we write unit vectors as 
$\xi_{\rho g} \in H_{\rho g}$ and $\xi_{g \rho} \in H_{g \rho}$.
Therefore the $q$-fundamental solution $\R$ in $\rep G_{\kp}$ is described by the vectors
\[
\sqrt{2} \xi_{g \rho} \otimes \xi_{\rho g}, 
\]
and
\[
\frac{1}{\sqrt{2}} \sum_{g \in K_4} \xi_{\rho g} \otimes  \xi_{g \rho}.
\]
The vector spaces associated with the the $\rep SU_{-1}(2)$-module category $\hilb_f$ 
are $H_{\rho}$ and $H_g$ of dimensions 2 and 1.
Here the unital maps of the $\rep SU_{-1}(2)$-module homomorphisms are expressed as
\[
\psi_g \colon H_{\rho} \otimes H_{\rho g} \to H_{1/2} \otimes H_{g}
\]
for $g \in K_4= \{e,a,b,c \}$ and 
\[
\psi_{\rho} \colon \bigoplus_{g \in K_4} H_g \otimes H_{g \rho} \to H_{1/2} \otimes H_{\rho}
\]
satisfying the following commutative diagrams:\\
\begin{figure}[h]
\begin{minipage}[b]{0.45\textwidth}
\centerline{
\xymatrix@=1em{ H_g \ar[rr]^>>>>>>{\text{id} \otimes \R}  \ar[ddrr]_{\R \otimes \text{id}}& 
&H_g \otimes H_{g \rho} \otimes H_{\rho g} \ar[dd]^{(\text{id} \otimes \psi_g)(\psi_{\rho} \otimes \text{id})}
\\ &&
\\ & & H_{1/2} \otimes H_{1/2} \otimes H_g,
}}
\end{minipage}
\hspace{0.3cm}
\begin{minipage}[b]{0.45\textwidth}
\centerline{
\xymatrix@=1em{ H_{\rho} \ar[rr]^>>>>>>{\text{id} \otimes \R}  \ar[ddrr]_{\R \otimes \text{id}}& 
& \bigoplus_{g \in K_4} \ H_{\rho} \otimes H_{\rho g} \otimes H_{g \rho} 
\ar[dd]^{(\text{id} \otimes \psi_{\rho})(\psi_g \otimes \text{id})}
\\ &&
\\ & & H_{1/2} \otimes H_{1/2} \otimes H_{\rho}.
}}
\end{minipage}
\end{figure}

From the projections $P_i \in M_4(\mathbb{C})$ in 
the tensor product of fundamental representation $U$ of $G_{\kp}$ with itself
$U \otimes U = \sum_{i=1}^4 p_i \otimes u_i$, 
we can compute the maps $\psi_g$, $\psi_\rho$ concretely.

Let $\{ \xi_1, \xi_2 \}$ be an orthonormal basis of $H_{\rho}.$

\begin{thm}
The unitary maps 
\begin{align*}
\psi_e &= 
 \begin{pmatrix}
 0 & 1 \\ 
 1 & 0
 \end{pmatrix},
\psi_a = \frac{1}{\sqrt{2}}
 \begin{pmatrix}
 1 & -1 \\ 
 1 & 1
 \end{pmatrix},
 \psi_b = 
 \begin{pmatrix}
 1 & 0 \\ 
 0 & -1
 \end{pmatrix},  
\psi_c = \frac{1}{\sqrt{2}}
 \begin{pmatrix}
 1 & 1 \\ 
 -1 & 1
 \end{pmatrix}, \\ 
 \psi_\rho &= \frac{1}{2}
 \begin{pmatrix}
 0 & 1 & \sqrt{2} & 1 \\
 \sqrt{2} & 1 & 0 & -1 \\
 \sqrt{2} & -1 & 0 & 1 \\
0 & 1 & \sqrt{2} & 1 \\
 \end{pmatrix}
\end{align*} 
associated with the $\rep SU_{-1}(2)$-module homomorphism $\rep G_{\kp} \to \hilb_{\text{f}}$ 
make the above diagrams commutes.
Here the matrix presentation of $\psi_\rho$ is with respect to the basis 
\[
\{ \xi_e \otimes \xi_{e \rho}, \ \xi_a \otimes \xi_{a \rho}, \ \xi_b \otimes \xi_{b \rho}, \ \xi_c \otimes \xi_{c \rho} \}
\]
of $ \bigoplus_{g \in K_4} H_g \otimes H_{g \rho}$ and the basis 
\[
\{ e_1 \otimes \xi_1, \ e_1 \otimes \xi_2,\  e_2 \otimes \xi_1, \ e_2 \otimes \xi_2 \}
\]
of $H_{1/2} \otimes H_{\rho}$.
\end{thm}

\begin{proof}
We show that the equation
\begin{align} \label{psi_eq}
( \id \otimes \psi_g)(\psi_\rho \otimes \id)(\id \otimes \R )(\xi_g) = (\R \otimes \id)(\xi_g)
\end{align}
holds for the unit vector $\xi_g$ in $H_g$ in the case $g=b$. 
On the right hand side we have
\begin{align*}
(\R \otimes \id)(\xi_b)
= (e_1 \otimes e_1 + e_2 \otimes e_2) \otimes \xi_b,
\end{align*}
while on the left hand side we have
\begin{align*}
( \id \otimes \psi_b)(\psi_\rho \otimes \id)(\id \otimes \R )(\xi_b)
&= \sqrt{2}( \id \otimes \psi_b)(\psi_\rho \otimes \id)(\xi_b \otimes \xi_{b \rho} \otimes \xi_{\rho b}) \\
&= \sqrt{2}( \id \otimes \psi_b) \left( \frac{1}{\sqrt{2}} (e_1 \otimes \xi_1 - e_2 \otimes \xi_2) \right) \\
&= e_1 \otimes e_1 \otimes \xi_b + e_2 \otimes e_2 \otimes \xi_b.
\end{align*}
Therefore, (\ref{psi_eq}) for $g=b$ holds. Other cases can be shown similarly.
\end{proof}

\end{document}